\documentclass[reqno]{amsart}

\usepackage{amsmath,amssymb,amsthm}

\usepackage{tikz}

\usepackage{caption}
\usepackage{graphicx}
\usepackage{graphics}

\newtheorem{theorem}{Theorem}

\newtheorem{lemma}{Lemma}[section]

\newtheorem{proposition}[lemma]{Proposition}
\newtheorem*{definition}{Definition}

\title[]{Regular Structures in Kronecker Permutations}
\author[]{François Clément}
\address{Department of Mathematics, University of Washington, Seattle}
\email{fclement@uw.edu} 

\begin{document}
\begin{abstract}
     Kronecker sequences $(k \alpha \mod 1)_{k=1}^{\infty}$ for some irrational $\alpha > 0$ have played an important role in many areas of mathematics. It is possible to associate to each finite segment $(k \alpha \mod 1)_{k=1}^{n}$ a permutation $\pi \in S_n$ associated with the canonical lifting to two dimensions. We show that these permutations induced by Kronecker sequences based on irrational $\alpha$ are extremely regular for specific choices of $n$ and $\alpha$. In particular, all quadratic irrationals have an infinite number of choices of $n$ that lead to permutations where no cycle has length more than 4. 
\end{abstract}
\maketitle

\section{Introduction and Results}
\subsection{Introduction}
Let $\alpha > 0$ be irrational. The associated Kronecker sequence, also known as irrational rotation on the torus, is an infinite sequence defined by 
$$ x_k \equiv k \alpha \mod 1.$$
We also use the fractional part notation $x_k = \left\{k \alpha \right\}$, where 
$\left\{x \right\} = x - \left\lfloor x \right\rfloor$.

\begin{figure}
    \centering
    \includegraphics[width=0.25\linewidth]{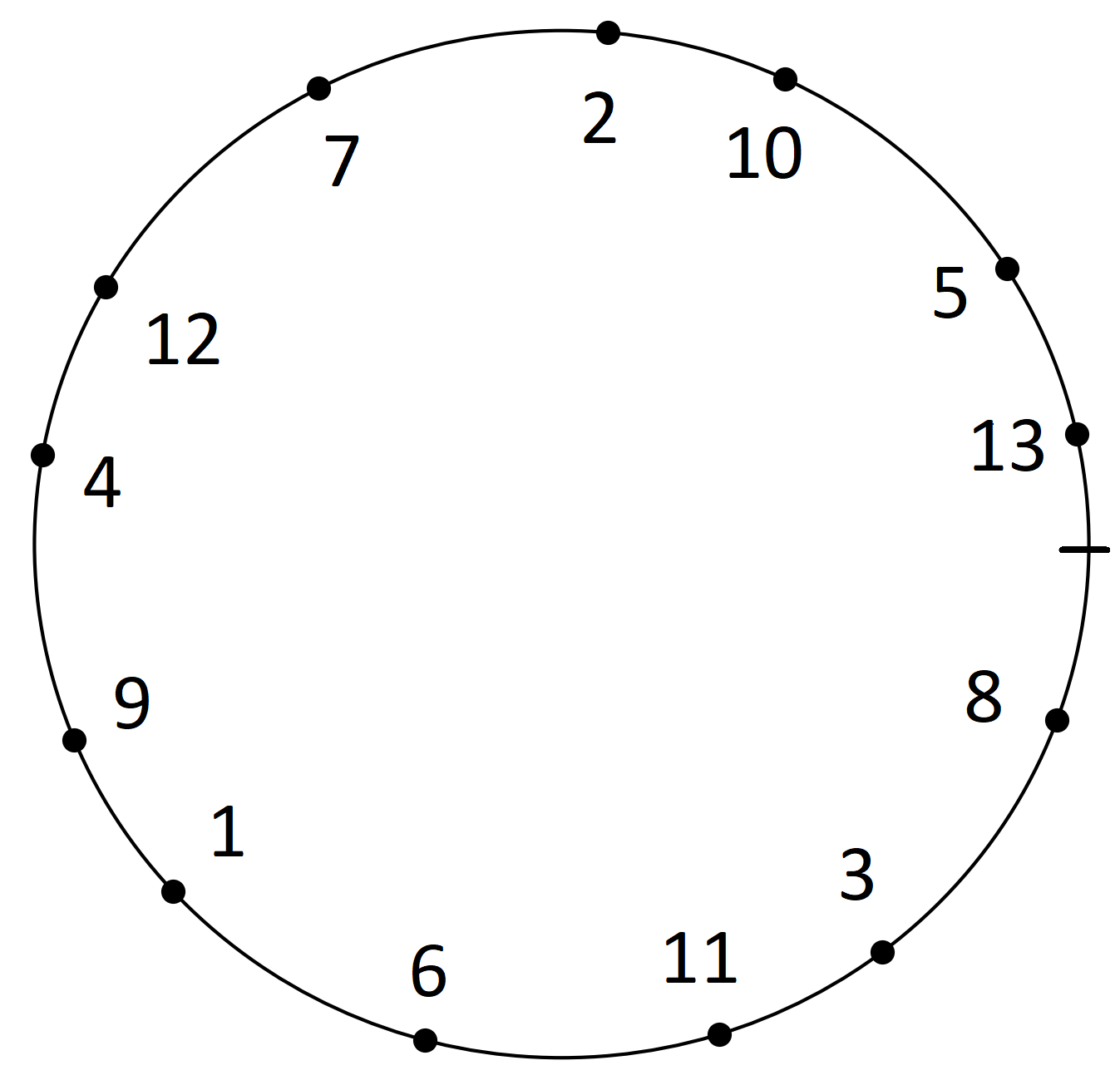}
    \caption{The first 13 points of the Kronecker sequence $\left\{ \varphi k \right\}$ on the torus, where $\varphi$ is the golden ratio}
    \label{fig:Kro}
\end{figure}

Its structure on the torus is very well-understood, in particular thanks to the very well-known three-gap theorem, or Steinhaus problem, shown initially by Sós~\cite{Sos} and with alternative proofs and generalizations in~\cite{ Fried, Haynes2014, HaynesMarklof,Marklof2017,Slater_1967, vanRavenstein_1988}. It states that the gap between consecutive points (for example 1 and 6 in Figure~\ref{fig:Kro}), can only take 3 specific values.
The regularity of the Kronecker sequence as a sequence on $\mathbb{S}^1$ can be turned into a very regular set on $[0,1]^2$ by considering the set $(k/n, \left\{ \varphi k \right\})_{k=1}^{n}$. This specific set has become increasingly central in a number of problems \cite{BilykTY12,CDKP, PNAS, Hinrichs2014, Nagel},  and there are now indications that it may exhibit a type of `universality' in the sense that it shows up naturally as the minimal energy configuration for a large number of different energies (see for example~\cite{Hinrichs2014}).

\subsection{Fibonacci Permutations.} 
We were motivated by the following simple example. Given a set like $(k/13, \left\{ \varphi k \right\})_{k=1}^{13}$ in Figure 2 and considering its pattern, there is a natural permutation associated to a Kronecker sequence. For a any finite segment $(k \alpha \mod 1)_{k=1}^{n}$, we may order these $n$ real numbers in $[0,1]$ by size and define a permutation 
$\sigma:\left\{1,2,\dots, n\right\} \rightarrow \left\{1,2,\dots,n\right\}$ so that
$$ x_{\sigma(1)} < x_{\sigma(2)} < \dots < x_{\sigma(n)}.$$

If $\alpha$ is irrational, then the $n$ real numbers are distinct. Returning to the initial example illustrated in two dimensions in Figure~\ref{fig:2}, we see that the 13-th element is the smallest and the 5-th element is the second smallest while the 8-th element is the largest, leading to the permutation
$$\sigma = \left(
 \begin{array}{ccccccccccccc}
1 & 2 & 3 & 4 & 5 & 6 & 7 & 8 & 9 & 10 & 11 & 12 & 13 \\
13 & 5 & 10 & 2 & 7 & 12 & 4 & 9 & 1 & 6 & 11 & 3 & 8
\end{array}
\right).
$$

\begin{figure}[h]
    \centering
    \includegraphics[width=0.35\linewidth]{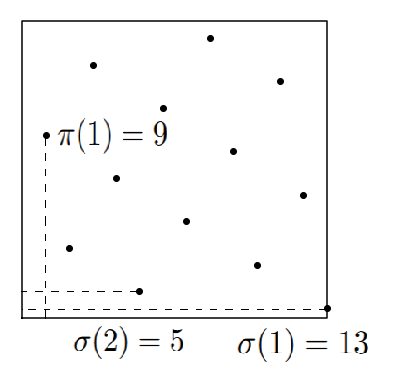}
    \caption{Taking Figure 1, we see that the closest element to 0 is the thirteenth. In two dimensions, this corresponds to looking at the smallest $y$-coordinate: we get $\sigma(1)=13$. Then $\sigma(2)=8$, and one can work their way up the $y$-coordinates of the points to get the permutation. Doing similarly but with the $x$-axis gives the inverse permutation $\pi$. $x_1$ is the ninth greatest element, therefore $\pi(1)=9$.}
    \label{fig:2}
\end{figure}
A natural way to decompose permutations is to decompose them into cycles. Here, the decomposition takes the form
$$ \sigma = (7~ 4~ 2~ 5) (10~ 6~ 12~ 3) (11) (13~ 8~ 9~ 1)$$
which is a decomposition into three 4-cycles and one fixed point. We believe these permutations to be quite natural and will refer to them as Kronecker Permutations. A formal definition is as follows. 
\begin{definition} Let $\alpha > 0$ be irrational and $n \in \mathbb{N}$. The associated Kronecker permutation $\sigma \in S_n$ has the property that $\pi(i) \alpha \mod 1$ is the $i-$th smallest number among the $n$ numbers $k\alpha \mod 1$ for $1 \leq k \leq n$.
\end{definition}

In practice, $\pi:=\sigma^{-1}$ is easier to work with than $\sigma$ itself. Since our results describe cycle lengths as well as fixed points, we will work with $\pi$ as the results naturally carry over to $\sigma$. Results will be stated with $\sigma$ as it is the more natural object, but we will be working with $\pi$ in the proofs.
Given this, we refer to the Kronecker Permutation and its inverse permutation associated with $n$ points as $\sigma_n$ and $\pi_{n}$ respectively, shortened to $\sigma$ and $\pi$ whenever the number of points is clear.\\

Kronecker Permutations are highly structured for suitable choices of $n$ and $\alpha$.
As is frequently the case in the study of irrational rotations on the Torus, the case
of $\alpha = \varphi$ being the golden ratio is somewhat special. We can analyze the arising Kronecker permutations explicitly whenever $n$ is a Fibonacci number. In particular, the arising cycles in the cycle decomposition have uniformly bounded length.

\begin{theorem}\label{th:fibo} 
    Let $n\in \mathbb{N}_{> 0}$, let $F_n$ denote the $n-$th Fibonacci number, with $F_1=F_2=1$, and let $\sigma_{F_n}$ denote the permutation associated with $(k \varphi \mod 1)_{k=1}^{F_n}$.
    \begin{enumerate}
        \item If $n$ is even, $\sigma_{F_n}$ is comprised of 2-cycles and fixed points.
        \item If $n \in \left\{1,5\right\} \mod 6$ is comprised of cycles of length 4 and one fixed point.
    \item If $n \equiv 3 ~ \mod 6$, then $\sigma_{F_n}$ is comprised of cycles of length 4 and one cycle of length 2.
    \end{enumerate}
\end{theorem}

We believe Theorem~\ref{th:fibo} to be one way to understand the origin of the incredible regularity of the Fibonacci set. This set is particularly well-distributed for $F_n$ many points, and the permutation has a very precise structure. For a non-Fibonacci number of points $k$, the point set is less uniform, and the associated permutation has longer cycles. For example, for $k=36$ to $51$, there are never two cycles of identical length, varying from a single cycle of length 48 and a fixed point for $k=49$, to 4 cycles of lengths 4, 5, 14 and 26 plus a fixed point for $k=50$.

\subsection{General result.}
Theorem~\ref{th:fibo} relies primarily on the properties of the continued fraction expansion of $\varphi$. We can extend this result to other irrationals with a periodic continued fraction expansion which, in particular, include the roots of quadratic polynomials with integer coefficients.

\begin{theorem}\label{th:quad}
    For any quadratic irrational with continued fraction expansion $$ \alpha = [a_0;\overline{a_1,\ldots,a_k}],$$
    let $j \equiv k-1 \mod k$ and consider the approximate continued fraction expansion $p_j/q_j$. Then the cycle decomposition of the permutation $\sigma_{q_j}$ corresponding to $\left\{ \alpha k \mod 1 \right\}_{k=1}^{q_j}$ is composed of
     \begin{enumerate}
     \item only fixed points and 2-cycles or
     \item 4-cycles and at most two fixed points or a single 2-cycle.
 \end{enumerate}
\end{theorem}
 
The arguments are somewhat flexible and allow us to prove a number of slightly stronger but also slightly more technical results, explained in Section~\ref{sec:lemm}.

\subsection{The general phenomenon.}
Kronecker sequences `naturally rotate around' and this should naturally lead to some structure of the induced permutations.  One could wonder how general that phenomenon is. We illustrate this with two canonical examples: $e$ and $\pi$. The continued fraction expansion of $e$ is known to be
$$ e = [2; 1,2,1,1,4,1,1,6,1,1,8,1,1,10,\dots]$$
 which is not regular but highly structured. The continued fraction approximations
$$ e \sim 2,3,\frac{8}{3},\frac{11}{4},\frac{19}{7},\frac{87}{32},\frac{106}{39},\frac{193}{71},\frac{1264}{465},\frac{1457}{536}, \dots$$
lead to curious permutation patterns. $\sigma_{71}$ decomposes into five 14-cycles and a fixed point, $\sigma_{465}$ has cycles of length 3, 6 and 30, $\sigma_{536}$ has 8 cycles of length 66 and 8 fixed points. There is clearly \textit{some} type of structure but it is less clear what exactly to expect; this may be an interesting avenue for further research. The continued fraction expansion of $\pi$, in contrast, is not known to have any type of regular structure whatsoever
$$ \pi = [3; 7, 15, 1, 292, 1, 1, 1, 2, 1, 3, 1, 14, 2, 1, 1, \dots].$$
The frequency of its digits appears, empirically, to follow the Gauss-Kuzmin law which is a way of saying that $\pi$, with respect to the continued fraction expansion, behaves like a generic real number. The first continued fraction approximants are
$$ \pi \sim 3, \frac{22}{7}, \frac{333}{106}, \frac{355}{113}, \frac{103993}{33102}, \frac{104348}{33215}, \dots$$
The permutations associated to $(\pi k \mod 1)_{k=1}^{q_j}$ appear to have some structure but the cycles need no longer be short. For example, $\sigma_{113}$ is given by sixteen 7-cycles and a fixed point, $\sigma_{33102}$ has cycles of length 54 and 1836. $\sigma_{33215}$ is made up of cycles of length 1, 2, 4, 6, 12 and 72. This leads to many natural questions about the behavior of the Kronecker permutations for general structured and `unstructured' irrationals. In particular, if we model the continued fraction expansion of a `generic' irrational by creating a sequence of continued fraction coefficients by taking independent samples from the Gauss-Kuzmin distribution, what can be said about the cycle structure of the Kronecker Permutations of such a `random' real number?

\subsection{Related Results.}\label{sec:motiv}
The Kronecker sequence has played a central role in mathematics \cite{ Cassels,Koksma_Unif,KuiNie,Nie92} and is intimately connected to the continued fraction expansion \cite{Bugeaud_2012}. In particular, they give rise to a type of sequence with optimal `regularity' on $\mathbb{S}^1$ (Schmidt's Theorem \cite{Schmidt}). The problem of finding point sets with good distribution in higher dimension is unsolved, see~\cite{Pano} for a complete overview, or~\cite{BilykSmall} for the most recent improvement. It was recently observed~\cite{CDKP,MPMC} that there exist point sets \emph{much} more regular than our previous best constructions. More importantly, their structure is very different from past constructions, suggesting that we may have missed a completely different class of constructions that could improve discrepancy bounds in higher dimensions. \cite{PNAS} suggests that relying on permutations, and in particular the Fibonacci permutation in two dimensions, could help reduce the computational burden and explore these constructions.
Separately, it is conjectured that the Fibonacci set (see Fig. 1) may play a very special role in the calculus of variations on $\mathbb{T}^2$. Minimizing a notion of energy over configurations of $n$ points on $\mathbb{T}^1 \cong \mathbb{S}^1$ will often lead to $n$ uniformly distributed points. No analogous `canonical' set exists on $\mathbb{T}^2$. In the last decade, the Fibonacci set has been proposed as a type of `canonical' optimal point set for tensor product energies on $\mathbb{T}^2$, such as the periodic $L_2$ discrepancy~\cite{BilykTY12, Hinrichs}. More recently, an optimization approach in~\cite{Nagel} showed its optimality for a wider class of functions, albeit for only a few points. The regularity phenomenon reported in our paper was discovered by accident when investigating these questions.

\section{Lemmas and Technical Results}\label{sec:lemm}

\subsection{Continued Fraction Expansions}

This section contains a basic introduction of continued fraction expansions, as well as all the related results that we will require. The reader can find a much more thorough description in~\cite{Bugeaud_2004}.
To begin, any real $\alpha$ can be described by a (possibly infinite) sequence of integers $[a_0;a_1,\ldots,a_k,\ldots]$. This sequence is such that
$$\alpha=a_0+\frac{1}{a_1+\frac{1}{a_2+\frac{1}{a_3+\ldots}}}.$$
This sequence is called the continued fraction expansion of $\alpha$, and each $a_k$ is the $k$-th convergent of $\alpha$. Rationals correspond exactly to the numbers with a finite continued fraction expansion. Since each prefix of the sequence $[a_0;a_1,\ldots,a_k]$ is finite, this corresponds to a rational, the $k$-th partial quotient, which is written $p_k/q_k$ where both $p_k$ and $q_k$ are coprime integers.

\begin{lemma}[Theorem 1.3 in~\cite{Bugeaud_2004}]\label{lemma:recursive}
    Setting
    $p_{-1}=1,\mbox{~} p_{0}=a_0, q_{-1}=0,\mbox{~and~} q_{0}=1,$
    we have, for any $n \in \mathbb{N}_{>0}$,
    $$p_n=a_{n}p_{n-1}+p_{n-2}\mbox{~and~}q_n=a_{n}q_{n-1}+q_{n-2}.$$
\end{lemma}
As done in the proof of Theorem 2.1 in~\cite{AdamBug}, these relations can be written out in matrix form to obtain the following result.

\begin{lemma}\label{lemma:palindrome}
    Let $\beta$ be an irrational between $0$ and 1, with continued fraction expansion $\beta=[0;a_1,a_2,\ldots]$. Then $p_{n-1}=q_{n}$ if and only if $a_0,\ldots,a_n$ is a palindrome.
\end{lemma}
 We add the proof of this lesser-known result for the convenience of the reader.
\begin{proof}
    Using a framework of Frame~\cite{Frame}, the matrix 
    $$M:=\begin{pmatrix}
q_n & q_{n-1}\\
p_n & p_{n-1}
\end{pmatrix}$$ containing the numerators and denominators of the $n-1$-th and $n$-th partial quotients can be decomposed in
$$M=\begin{pmatrix}
a_1 & 1\\
1 & 0
\end{pmatrix}\begin{pmatrix}
a_2 & 1 \\
1 & 0 
\end{pmatrix}\ldots\begin{pmatrix}
a_n & 1 \\
1 & 0 
\end{pmatrix}$$
using the $n$ first convergents. Using the mirror formula (see~\cite{ADAMCZEWSKI2007}, or an older version under the name of folding formula~\cite{vdPoorten}), we have that
$q_{n-1}/q_n=[0;a_n,\ldots,a_1]$
$M$ is symmetric if and only if $p_{n}=q_{n-1}$. However, we also have that $p_n/q_n=[0;a_1,\ldots, a_n]$
by definition of the continued fraction expansion. Since the continued fraction expansion is unique, we get that $a_1,\ldots, a_n$ is a palindrome.
\end{proof}

There also exists a general relation between the $p_n$ and $q_n$.
\begin{lemma}[Theorem 1.4 in~\cite{Bugeaud_2004}]\label{lemma:altern}
    For any $n>0$, $p_nq_{n-1}-p_{n-1}q_{n}=(-1)^n$.
\end{lemma}
This result can be showed using the matrix expression that led to Lemma~\ref{lemma:palindrome} and considering the determinant. 
Furthermore, these partial quotients correspond to the sequence of the best possible approximations to $\alpha$: each $p_n/q_n$ is a better approximation than any other $a/b$ where $b\leq q_n$. It is also known that the sequence $p_n/q_n$ alternates in approximating $\alpha$ from above and from below.
\begin{lemma}[Lemma 1.2 in~\cite{Bugeaud_2004}]\label{lemma:even}
    For $n>0$ and $\alpha$ irrational
    $$\alpha > \frac{p_n}{q_n} $$
     if and only if $n$ is even.
\end{lemma}
 The final Lemma required is a bound on the error each $k$-th partial quotient approximation is making, shown by Hurwitz in~\cite{Hurwitz1891}.

\begin{lemma}\label{lemma:Hurwitz}
    Let $\alpha$ be an irrational with partial quotient $p/q$. Then
    $$\left|\alpha -\frac{p}{q}\right|\leq \frac{1}{\sqrt{5}q^2}.$$
\end{lemma}

\subsection{Useful results} Building on these continued fraction expansion properties, we will need a couple of results to prove Theorem~\ref{th:fibo}. The first proposition is a generalization of Theorem~\ref{th:quad}. Irrationals for which the continued fraction expansion is `often' palindromic without being periodic are difficult to write down in closed form, as such we prefer to focus on quadratic irrationals in the main result. 

\begin{proposition}\label{th:cycles}
 Let $\alpha$ be an irrational in $[0,1)$ such that the first $n$ terms of its continued fraction expansion $[0;a_1,\ldots,a_n]$ form a palindrome. Then, given $p_n/q_n$ the $n$-th convergent of $\alpha$, the permutation induced by $(\{k\alpha\})_{k \in \{1,\ldots,q_n\}}$, $\pi_{q_n}$, is composed of one of the two following options:
 \begin{enumerate}
     \item If $n$ is even, fixed points and 2-cycles.
     \item If $n$ is odd, 4-cycles with at most two fixed points or a single 2-cycle.
 \end{enumerate}
\end{proposition}

We point out here that the $0$ from the continued fraction expansion is not included in the palindrome. This allows us to consider any $\alpha$ corresponding to a quadratic irrational modulo 1. Indeed, it is known from results of Lagrange~\cite{Lagrange} and Euler~\cite{Euler} that quadratic irrationals are exactly the irrationals whose continued fraction expansion can be written $[a_0,\overline{a_1,\ldots,a_n}]$, where $a_1,\ldots,a_n$ repeats periodically and is a palindrome. Removing $a_0$ gives the palindrome structure, directly implying Theorem~\ref{th:quad}. Since we are only considering fractional parts modulo 1, we will shorten $[0,a_1,\ldots,a_k]$ to $[a_1,\ldots,a_k]$ without loss of generality. Continued fraction expansions such as $[\overline{r}]$ should be understood as $[a_0,\overline{r}].$\\

Furthermore, we are able to provide a stronger result for irrationals whose partial quotients are constant, by also describing the fixed points of the permutation.

\begin{lemma}\label{th:fixed}
    Let $\alpha$ be an irrational with continued fraction expansion $[\overline{r}]$ and $n$ be even. Let $p_n/q_n$ be the $n$-th convergent of $\alpha$, with $p_0=0$, $p_1=q_0=1$ and $q_1=r$.
    \begin{itemize}
        \item[1.]If $n/2$ is odd, some of the fixed points are the multiples of $q_{n/2}$ in $\{1,\ldots,q_n\}$.
        \item[2.] If it is even, some of them are instead either the multiples in $\{1,\ldots,q_n\}$ of $\frac{r}{2}q_{n/2}+q_{{n/2}+1}$ when $r$ is even or those of $q_{n/2-1}+q_{n/2+1}$ for r odd.
    \end{itemize}
\end{lemma}

We note that these appear to be the only fixed points for all candidate irrationals whose continued fraction expansion is $[\overline{r}]$, such as $\{\varphi\},\{\sqrt{2}\},\{\sqrt{10}\},\{\sqrt{17}\}$ or $\{\sqrt{26}\}$. We believe these fixed points to be the only ones, but we have not been able to prove that there are no others.

To illustrate Lemma~\ref{th:fixed}, consider $\sqrt{2}$, whose continued fraction expansion is $[1;\overline{2}]$. The $q_n$ are given by the Pell numbers: $1,2,5,12,29,70,\ldots$. For the permutation $\pi_{q_6,\{\sqrt{2}\}}$ for $q_6=70$, the fixed points are exactly the multiples of 5, which corresponds to $q_3$. For $\pi_{q_4,,\{\sqrt{2}\}}$, the fixed points are the multiples of 3, which is precisely $\frac{2}{2}q_2+q_1$.

\section{Proofs}

\subsection{Outline.} We begin by proving Proposition ~\ref{th:cycles}. The first argument is to approximate our point set by another which has the same permutation, but uses a rational quotient $p_n/q_n$ rather than an irrational $\alpha$ to define the Kronecker set. We then show that for this set defined by the rational quotient, the permutation $\pi$ has a `simple' expression. This expression is used to show that $\pi \circ \pi$ (i.e. $\pi^2$) is the identity permutation for $n$ even, and that $\pi^4$ is the identity for $n$ odd. In this last case, we can also characterize which elements are not in a cycle of length four, concluding the proofs of Proposition~\ref{th:cycles} and Theorem~\ref{th:quad}.
We then describe ``large'' families of fixed points in the $n$ even case, before concluding with the proof of Theorems~\ref{th:fibo}, which only requires one extra element compared to the previous results: there cannot be two fixed points and cycles of length four at the same time.

\subsection{An exchange Lemma.}
We begin by introducing an approximation of our point set. For any irrational $\alpha$, instead of considering the Kronecker sequence $(\{k\alpha\})_{k \in \mathbb{N}}$, we will consider an approximation based on its partial quotients $p_n/q_n$. For the permutation induced by the first $q_n$ points, we replace $\alpha$ by $p_n/q_n$, with the following modification for the last point: $\pi_{q_n}(q_n)=q_n$ if $n$ is even and 1 if odd. The idea is illustrated in Figure~\ref{plots}: the last point of $(\{kp_n/q_n\})_{k=1}^{q_n}$ is always equal to 0 modulo 1, while the last one of  $(\{k\alpha\})_{k=1}^{q_n}$ can be either close to 0 from below or from above modulo 1.

\begin{figure}[h]
    \centering
    \includegraphics[width=0.45\linewidth]{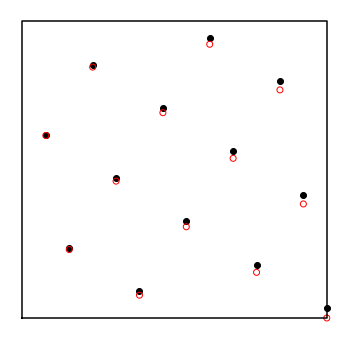}
    \hfill
    \includegraphics[width=0.45\linewidth]{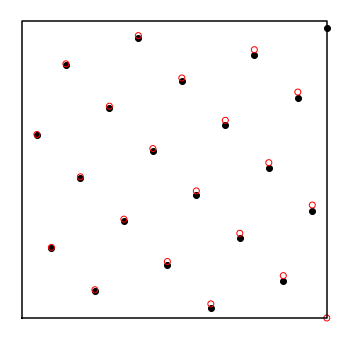}
    \caption{In empty red circles, the points of $(\{kp_n/q_n\})_{k=1}^{q_n}$ for $p_n/q_n=8/13$ (left) and $p_n/q_n=13/21$ (right). In black, the points of $(\{k\alpha\})_{k=1}^{q_n}$. Both permutations are extremely similar, with identical relative positions for the first $q_n-1$ points. The only difference is the final point for $n$ even ($q_n=21$ is the eighth Fibonacci number).}
    \label{plots}
\end{figure}

\begin{lemma}\label{lem:approx}
    The permutations induced by $(k\alpha \mod 1)_{k=1}^{q_n-1}$ and $(kp_n/q_n \mod 1)_{k=1}^{q_n-1}$ are identical. Moreover, $\pi_{q_n}$ is identical to the permutation associated with $(k\alpha\mod 1)_{k=1}^{q_n}$.
\end{lemma}

\begin{proof}
    The argument relies on Lemma~\ref{lemma:Hurwitz}.
    With our approximation, the $k$-th point is shifted vertically by at most 
    $$\left|k\alpha -k\frac{p_n}{q_n}\right|\leq q_n\left|\alpha-\frac{pn}{q_n}\right|.$$
    Let $\varepsilon$ be this maximal vertical movement,
    $$\varepsilon:=q_n\left|\alpha -\frac{p_n}{q_n}\right|.$$
    
    Since every $j/q_n$ is obtained exactly once as $p_n$ and $q_n$ are coprime, a change in the permutation can only happen if $\varepsilon$ is at least half the vertical distance between two consecutive $j/q_n$ and $(j+1)/q_n$. Therefore, we cannot have a change in the permutation unless $\varepsilon>1/(2q_n)$.
    Using Lemma~\ref{lemma:Hurwitz}, we observe that $\varepsilon \leq 1/(\sqrt{5}q_n)<1/(2q_n)$. The first $q_n-1$ elements therefore keep the same relative position.
    It remains to tackle the issue of the last point. In our setting, Lemma~\ref{lemma:even} corresponds to $\{q_n\alpha\}$ being the maximal value of the first $q_n$ terms if $n$ is even, and the minimal value otherwise. Setting $\pi_{q_n}(q_n)=q_n$ if $n$ is even and 1 if odd makes both permutations equal and concludes the proof.
\end{proof}

Using this Lemma, we will from now on only consider the permutation induced by $(kp_n/q_n\mod 1)_{k=1}^{q_n}$. Note that for $\alpha=\varphi$, $p_n/q_n=F_{n-1}/F_n$, where $F_n$ is the $n-$th Fibonacci number. Given the modification made for the last point, when $n$ is even, for any $k \in \{1,\ldots,q_{n}-1\}, \pi_{q_n}(k)$ is given by $p_{n}k \mod q_n$ and $\pi_{q_n}(q_n)=q_n$. For $n$ odd, $\pi_{q_n}(k)$ is given by $p_{n}k+1 \mod q_n$ for $k<q_n$ and 1 otherwise. 

\subsection{Proof of Proposition~\ref{th:cycles}} 

\begin{proof}
Let $\alpha$ be an irrational such that the first $n$ terms of its continued fraction expansion $[a_1,\ldots,a_n]$ form a palindrome, and abbreviate $\pi_{q_n}$ by $\pi$. Lemma~\ref{lemma:palindrome} states that $p_n=q_{n-1}$ if and only if $[a_1,\ldots,a_n]$ is a palindrome. Combining this with Lemma~\ref{lemma:altern} gives $p_n^2 =(-1)^n \mod q_n$.

\begin{itemize}
    \item We first show that for $n$ even $\pi^2=\mbox{Id}$, the identity permutation. If $n$ is even, for $k<q_n$, $\pi(k)$ is equivalent to $p_nk \mod q_n$, as stated above. We therefore have that $\pi^2(k)=\pi(p_nk)\equiv p_n^2k \mod q_n$. Since $p_n^2\equiv 1 \mod q_n$, we have that $\pi^2(k)\equiv k\mod q_n$ and, since both are between $1$ and $q_n$, $\pi^2(k)=k$. By definition, we already know that $\pi(q_n)=q_n$, it is a fixed point. We obtain the desired result, $\pi$ can therefore only contain cycles of length two and fixed points.
    \item Similarly, for $n$ odd, we show that $\pi^4=\mbox{Id}$. If $n$ is odd, for $k \leq q_n$, $\pi(k) \equiv p_nk+1 \mod q_n$. Applying this four times gives us $\pi^4(k) \equiv p_n^4k+p_n^3+p_n^2+p_n+1 \mod q_n$. Since $p_n^2 \equiv -1\mod q_n$, the last four terms cancel. Using $p_n^4\equiv 1 \mod q_n$, we are left with, for all $k \leq q_n$, $\pi(k)=k$. We have that $\pi^4=\mbox{Id}$, $\pi$ can be decomposed into cycles whose length divides four.
\end{itemize}
It now remains to show that there can be only a single cycle of length 2 or less in the $n$ odd case. An element $1\leq k\leq q_n$ is part of a cycle of length 2 or less if and only if $p_n^2k+p_n+1\equiv k[q_n]$ (i.e. $\pi^2(k)=k$). Since $p_n^2 \equiv -1\mod q_n$, this is equivalent to $p_{n}+1 \equiv 2k\mod q_n$. Since $k$ is less than $q_n$, this equation has at most two solutions, equal to $(p_{n}+1)/2$ and $(p_n+1+q_n)/2$. Only these may be either fixed points or part of a cycle of length 2, this concludes the proof.
\end{proof}
Theorem~\ref{th:quad} is an immediate consequence of this Proposition, as quadratic irrationals $\alpha=[0;\overline{a_1,\ldots,a_k}]$ have a palindromic expansion (ignoring the first 0) for the $k-1$-th partial quotient, and then whenever adding another $k$ terms: $[0,a_1,\ldots,a_{k-1}]$, then $[0,a_1,\ldots a_{2k-1}]$ and so on.

\subsection{Proof of Lemma~\ref{th:fixed}}
We now turn our attention to the characterization of the fixed points of the permutation. We consider an irrational $\alpha \in [0,1]$ such that its continued fraction expansion is $[0;\overline{r}]$. Ignoring the first 0, for any $n$ we have $p_n=q_{n-1}$ as the associated continued fraction expansion is a sequence of $r$'s and therefore a palindrome. We can furthermore define $p_n$ and $q_n$ recursively with the following equations given in Lemma~\ref{lemma:recursive}
$$p_n=rp_{n-1}+p_{n-2}~\mbox{and}~q_n=rq_{n-1}+q_{n-2},$$
with $p_{0}=0$, $p_{1}=q_{0}=1$ and $q_1=r$ since $a_0=0$ for $\alpha \in [0,1)$ (the first terms are adapted to our setting where $a_0=0$).

\begin{proof}
In all the cases we consider, the argument relies on the characterization of $\pi$ for $n$ even: $\pi(k)\equiv p_nk \mod q_n$ for all $k\in \{1,\ldots,q_n\}$. We then find one specific fixed point and then obtain as a consequence that the multiples of this element in $\{1,\ldots,q_n\}$ are also fixed points.
\begin{enumerate}
\item    We begin with the odd $n/2$ case, where we show that the fixed points are the multiples of $q_{n/2}$ regardless of the choice of $r$. As stated previously as a consequence of Lemma~\ref{lem:approx}, we have that $\pi(k) \equiv p_nk \mod q_n$. We therefore have that $\pi(p_{n/2})\equiv p_{n/2}p_{n} \mod q_n$. We wish to show that $q_{n/2}p_{n}-q_{n/2-1}q_{n}=q_{n/2}$, which will imply $q_{n/2}p_{n} \equiv q_{n/2} \mod q_n$ and give the desired fixed points.
\begin{align*}
    q_{n/2}p_{n}-q_{n/2-1}q_{n}&=q_{n/2}q_{n-1}-q_{n/2-1}q_{n}\\
    &=q_{n-1}(rq_{n/2-1}+q_{n/2-2})-q_{n/2-1}(rq_{n-1}+q_{n-2})\\
    &=q_{n-1}q_{n/2-2}-q_{n/2-1}q_{n-2}\\
    &=q_{n-3}q_{n/2-2}-q_{n/2-3}q_{n-2}
\end{align*}

Since $n/2$ is odd, we can repeat these steps $(n-1)/2$ times to obtain,
$$q_{n/2}p_{n}-q_{n/2-1}q_{n}=q_1q_{n/2}-q_0q_{n/2+1}=q_{n/2}.$$
This implies that all multiples of $q_{n/2}$ in $\{1,\ldots,q_n\}$ are fixed points of $\pi$ for $n/2$ odd.
\item We now turn to the even $n/2$ case. We need to distinguish the $r$ odd and even cases, though the reasoning will be similar in both cases. We begin with the $r$ odd case, we wish to show that $q_{n+1}+q_{n-1}$ is a fixed point (and therefore its multiples in $\{1,\ldots, q_n\}$ as well). We have
\begin{align*}
    q_{n/2+1}p_{n}-q_{n/2}q_{n}&=q_{n/2+1}q_{n-1}-q_{n/2}q_{n}\\
    &=q_{n/2-1}q_{n-1}-q_{n/2}q_{n-2}\\
    &=q_{n/2-1}q_{n-3}-q_{n/2-2}q_{n-2}
\end{align*}
After $n/4$ steps like the one above, we obtain $q_{1}q_{n/2-1}-q_0q_{n/2}=q_{n/2-1}$. A similar reasoning gives $q_{n/2-1}p_{n}-q_{n}q_{n/2-2}=q_{n+1}$. Summing both relations together gives
$$\left(q_{n/2+1}+q_{n/2-1}\right)p_n \equiv q_{n/2+1}+q_{n/2-1} \mod q_n.$$
Using the characterization of $\pi$, $\pi(k)\equiv p_nk \mod q_n$, we obtain the desired fixed points.
\item Notice that this result is true both for $r$ even and odd. However, in the even case, $rq_{n/2}/2+q_{n/2-1}$ is also a fixed point.
Indeed, with the same method as previously $q_{n/2}p_n-q_{n}q_{n/2-1}=-q_{n}$.
We obtain $$\left(\frac{r}{2}q_{n/2}+q_{n/2-1}\right)p_n\equiv q_{n/2+1}-\frac{r}{2}q_{n/2} \mod q_n.$$
Replacing $q_{n/2+1}$ by $rq_n+q_{n/2-1}$ in the right-hand side of the equation above and using again the characterization of $\pi$ shows that $rq_{n/2}/2+q_{n/2-1}$ is a fixed point. Since $q_{n/2+1}+q_{n/2-1}=2(rq_{n/2}/2+q_{n/2-1})$, the multiples of $rq_{n/2}/2+q_{n/2-1}$ contain all those of $q_{n/2+1}+q_{n/2-1}$. We obtain that in the $r$ even case, the multiples of $rq_{n/2}/2+q_{n/2-1}$ are fixed points.
\end{enumerate}
\end{proof}
We conclude with the proof of Theorem~\ref{th:fibo}, which only requires showing that there is only a single fixed point in the odd $n$ and even $q_n$ case.

\subsection{Proof of Theorem~\ref{th:fibo}}
\begin{proof}
Firstly, the continued fraction expansion of $\{\varphi\}$ is $[0;\overline{1}]$. Using Proposition~\ref{th:cycles}, we already know that:
\begin{enumerate}
    \item For $n$ even, $\pi_{F_n}$ has only cycles of length two and fixed points. Lemma~\ref{th:fixed} further describes (some of) these fixed points.
    \item For $n$ odd, $\pi_{F_n}$ has only cycles of length four and either one cycle of length two or at most two fixed points.
\end{enumerate}
We now consider only this last $n$ odd case, and show that there cannot be two fixed points. This then implies that the parity of $F_n$ enforces either the presence of a single fixed point ($F_n$ odd) or a single cycle of length two ($F_n$ even). Using the relation $F_{n+1}=F_n+F_{n-1}$ and $F_1=F_2=1$, one can easily show by induction that $F_n$ is even if and only if $n$ is a multiple of 3. When $n$ is odd and not divisible by 3, there can only be a single fixed point by parity of $F_n$ (there cannot be three or more by Proposition~\ref{th:cycles}). This covers the cases $n\in\{1,5\} \mod 6$. There remains the case $n\equiv 3 \mod 6$, for which $F_n$ is even. In this case, the decomposition in cycles of length 4 leaves two elements. We now show that they necessarily form a cycle of length 2.

For this, we want to show that $\pi_{q_n}((F_{n-1}+1)/2)=(F_{n-1}+F_{n}+1)/2$, as they are the only two candidates that are not in a cycle of length four, as shown at the end of the proof of Proposition~\ref{th:cycles}. By the characterization of $\pi$ for $n$ odd, this corresponds to showing that $F_{n-1}(F_{n-1}+1)/2+1 \equiv (F_{n-1}+F_{n}+1)/2 \mod F_n$. For this, we use Cassini's identity:

$$(-1)^{n-1}=F_{n-2}F_{n}-F_{n-1}^2. $$
This leads to 
\begin{align*}
\pi\left(\frac{F_{n-1}+1}{2}\right)&=F_{n-1}\frac{F_{n-1}+1}{2}+1 =\frac{F_{n-1}^2+F_{n-1}}{2}+1\\
    &=\frac{F_{n-2}F_n -1+F_{n-1}}{2}+1=\frac{F_{n-2}-1}{2}F_n+\frac{F_{n}+F_{n-1}+1}{2}.
\end{align*}
Since $n\equiv 3 \mod 6$, $F_{n-2}-1$ is even, and we obtain the desired equality modulo $F_n$, showing that $\pi$ has a single cycle of length two and not two fixed points. This concludes the characterization of $\pi_{F_n}$ for any $n$.
\end{proof}

\bibliographystyle{plain}
\bibliography{refs}
\end{document}